\documentclass[12pt,reqno]{amsart}

\usepackage{amsmath,amssymb,amsthm}
\usepackage{hyperref}
\usepackage{xcolor}
\usepackage{fullpage}

\newtheorem{theorem}{Theorem}[section]
\newtheorem{lemma}[theorem]{Lemma}
\newtheorem{proposition}[theorem]{Proposition}

\theoremstyle{definition}
\newtheorem{definition}[theorem]{Definition}
\newtheorem{remark}[theorem]{Remark}
\newtheorem{question}[theorem]{Question}

\newcommand{\aitext}[1]{\textsf{#1}}

\hypersetup{
    colorlinks=true,
    linkcolor=black,          
    citecolor=black,          
    urlcolor={blue!70!black}, 
    pdfborder={0 0 0},
}

\usepackage{wasysym}

\newcommand\doublecheck{\textcolor{blue}{\checked\kern-0.6em\checked}}

\usepackage{tikz}
\usetikzlibrary{calc}
\usepackage[normalem]{ulem}  

\usepackage[breakable,skins]{tcolorbox}

\definecolor{humancolor}{RGB}{0, 102, 204}    
\definecolor{aicolor}{RGB}{180, 90, 50}       

\newlength{\marginbarwidth}
\newlength{\marginbarsep}
\newtcolorbox{humanparagraph}{%
  enhanced,
  breakable,
  blanker,                                    
  borderline west={\marginbarwidth}{-\marginbarsep}{humancolor},
  left=0pt,                                   
  right=0pt,
  top=2pt,
  bottom=2pt,
  before skip=0.5\baselineskip,
  after skip=0.5\baselineskip,
  overlay unbroken and first={
    \node[humancolor, font=\tiny\scshape, rotate=90, anchor=north]
      at ([xshift=-\marginbarsep-\marginbarwidth/2-12pt, yshift=-13pt]frame.north west) {Human};
  }
}

\newtcolorbox{aiparagraph}{%
  enhanced,
  breakable,
  blanker,
  borderline west={\marginbarwidth}{-\marginbarsep}{aicolor},
  left=0pt,
  right=0pt,
  top=2pt,
  bottom=2pt,
  before skip=0.5\baselineskip,
  after skip=0.5\baselineskip,
  overlay unbroken and first={
    \node[aicolor, font=\tiny\scshape, rotate=90, anchor=north]
      at ([xshift=-\marginbarsep-\marginbarwidth/2-12pt, yshift=-5pt]frame.north west) {AI};
  }
}

\newcommand{\humaninline}[1]{{\color{humancolor}\uwave{#1}}}
\newcommand{\aiinline}[1]{{\color{aicolor}\uwave{#1}}}

\newenvironment{authornote}{%
  \begin{quotation}
  \footnotesize
  \noindent\textsc{Author's note.}\enspace\ignorespaces
}{%
  \end{quotation}
}

\newenvironment{altabstract}{%
  \begin{quotation}
  \footnotesize
  \noindent\textsc{Abstract.}\enspace\ignorespaces
}{%
  \end{quotation}
}

\newcommand{\Mbar}{\overline{\mathcal{M}}}
\newcommand{\ee}{\mathbf{e}}
\newcommand{\ZZ}{\mathbb{Z}}
\newcommand{\QQ}{\mathbb{Q}}

\title[Extremal descendant integrals]{Extremal descendant integrals on moduli spaces of curves:
An inequality discovered and proved in collaboration with AI}

\author{Johannes Schmitt}
\address{Department of Mathematics, ETH Z\"urich, Switzerland}
\email{johannes.schmitt@math.ethz.ch}

\date{\today}

\begin{document}

\maketitle
\begin{altabstract}
    \begin{humanparagraph}
For the pure $\psi$-class intersection numbers $D(\textbf{e})=\langle \tau_{e_1} \cdots \tau_{e_n} \rangle_g$ on the moduli space $\Mbar_{g,n}$ of stable curves, we determine \aiinline{for} which choices of $\textbf{e}=(e_1, \ldots, e_n)$ the value of $D(\textbf{e})$ becomes extremal.
The intersection number is minimal for powers of a single $\psi$-class (i.e. all $e_i$ but one vanish), whereas maximal values are obtained for balanced vectors ($|e_i - e_j| \leq 1$ for all $i,j$). The proof uses the nefness of the $\psi$-classes combined with Khovanskii--Teissier log-concavity.
\end{humanparagraph}
\end{altabstract}

\begin{authornote}
\begin{humanparagraph}
The question of finding extremal values of the $\psi$-intersection  numbers first occurred to the author when looking for a toy problem to explore using the software OpenEvolve \cite{openevolve}. The conjecture that balanced exponents lead to the maximal values is a natural guess, and was indeed discovered quickly by the tested model. 
To the author's knowledge, this optimization-style problem was novel and not covered by existing literature: it is a simple and natural question, but somewhat \aiinline{orthogonal} to the questions usually studied in enumerative geometry.
After some experimental verification and presenting the conjecture to several colleagues (who confirmed its open status), it was submitted as a problem to the IMProofBench  project \cite{IMProofBench}. This project collects research level mathematics questions and tests them against a range of AI models. As part of this evaluation, the conjecture was independently proven by several such models, without human intervention (see Appendix \ref{sec:methods} for further details). 

The proof itself turns out to be unexpectedly simple, bypassing the rich structural results on descendant invariants on $\Mbar_{g,n}$, and instead using just a few basic properties of $\psi$-classes (nefness and symmetry) together with general results on intersection numbers of divisor classes.
As such, while the obtained theorem is a neat little result and original contribution to the literature, it would arguably be on the borderline of notability for a mathematical publication.

We still opted to use this opportunity and write the present note, both to present the result but in particular also to document the process that went into its discovery and subsequent write-up. Apart from showcasing the current level of sophistication of AI for these tasks, we also propose some best practices for attribution of AI use in mathematical writing (see Appendix \ref{sec:AI_attribution}). We apply these principles in the current paper, e.g. by marking human-generated text with a blue bar or \humaninline{underlines}, whereas AI-generated text is marked in brown.

After an initial draft of the paper was completed, we also added an alternative version of the proof of the main theorem (Section \ref{sec:formalized}), partially verified in Lean \cite{lean4}. The argument is split into a purely combinatorial optimization theorem, formalized by Claude Code and GPT 5.2 with minimal guidance by the author, and a geometric proposition proving the applicability of the optimization result to the intersection numbers above.
\end{humanparagraph}
\end{authornote}
\clearpage
\section{Introduction}\label{sec:intro}

\begin{aiparagraph}
Let $g, n \in \ZZ_{\geq 0}$ with $2g - 2 + n > 0$. The moduli space $\Mbar_{g,n}$ of stable curves of genus $g$ with $n$ marked points has dimension $d = 3g - 3 + n$. For each marked point $i$, the cotangent line bundle $\mathbb{L}_i$ at that marking defines the \emph{$\psi$-class} $\psi_i = c_1(\mathbb{L}_i) \in H^2(\Mbar_{g,n}, \QQ)$. The \emph{descendant integrals} (or \emph{intersection numbers of $\psi$-classes})
\[
\langle \tau_{e_1} \cdots \tau_{e_n} \rangle_g := \int_{\Mbar_{g,n}} \psi_1^{e_1} \cdots \psi_n^{e_n}
\]
are rational numbers that vanish unless $\sum_{i=1}^n e_i = d$. These integrals play a central role in the Witten--Kontsevich theorem \cite{Witten1991, Kontsevich1992}, which establishes that their generating function is a $\tau$-function for the KdV hierarchy.

Given $g$ and $n$, let
\[
E(g,n) = \left\{ \ee = (e_1, \ldots, e_n) \in \ZZ_{\geq 0}^n : \sum_{j=1}^n e_j = 3g - 3 + n \right\}
\]
be the set of valid exponent vectors. Define the function $D : E(g,n) \to \QQ$ by $D(\ee) = \langle \tau_{e_1} \cdots \tau_{e_n} \rangle_g$.

\begin{definition}
A vector $\ee \in E(g,n)$ is \emph{balanced} if $|e_i - e_j| \leq 1$ for all $1 \leq i, j \leq n$.
\end{definition}

Equivalently, writing $3g - 3 + n = an + b$ with $0 \leq b < n$ via Euclidean division, the balanced vectors are precisely the permutations of $(a, \ldots, a, a+1, \ldots, a+1)$ with $a$ appearing $n - b$ times and $a + 1$ appearing $b$ times.

\begin{theorem}[Extremal descendants]\label{thm:main}
Let $g \in \ZZ_{\geq 0}$ \humaninline{and $n \in \ZZ_{>0}$} with $2g - 2 + n > 0$.
\begin{enumerate}
    \item[\textup{(a)}] \textup{(Minimum)} The function $D$ achieves its minimum at the concentrated vector $(3g-3+n, 0, \ldots, 0)$ (or any permutation thereof), with value
    \begin{equation} \label{eqn:min}
        \langle \tau_{3g-3+n} \humaninline{\cdot \tau_0^{n-1} }\rangle_g = \frac{1}{24^g \, g!}.
    \end{equation}
    \item[\textup{(b)}] \textup{(Maximum)} The function $D$ achieves its maximum on a balanced vector $\ee \in E(g,n)$.
\end{enumerate}
\end{theorem}

The \humaninline{formula for the minimal value} in part (a) is well-known (\humaninline{see e.g. equation (5.36) in \cite{MR1180858}}); it also follows from the generating function identity
\begin{equation} \label{eqn:onepoint}
1 + \sum_{g \geq 1} t^g \int_{\Mbar_{g,1}} \psi_1^{3g-2} = \exp(t/24)
\end{equation}
established in \cite{FaberPandharipande2000}. Part (b) appears to be new.

In genus zero, part (b) follows immediately from the closed formula
\begin{equation}\label{eq:genus0}
\langle \tau_{e_1} \cdots \tau_{e_n} \rangle_0 = \frac{(n-3)!}{e_1! \cdots e_n!},
\end{equation}
valid when $\sum e_i = n - 3$. The multinomial coefficient on the right is maximized precisely when the $e_i$ are as equal as possible. For higher genus, no such closed formula exists, and the result requires the more sophisticated argument given in Section~\ref{sec:proof}.

\begin{remark}
The \humaninline{result above was also} verified computationally for all $0 \leq g \leq 11$ and $1 \leq n \leq 11$ using the SageMath \cite{sagemath} software package  \texttt{admcycles} \cite{admcycles}.
\end{remark}

\begin{remark}
Uniqueness need not hold: due to symmetry, all permutations of a balanced vector achieve the same value, and in some cases plateaus may occur where strict inequality fails.
\end{remark}
\end{aiparagraph}
\begin{humanparagraph}
Given the simple formula \eqref{eqn:min} for the minimal descendant, a natural question is:
\begin{question}
For  $3g-3+n = an + b$, does there exist a simple description (closed formula, generating function, recursion, $\ldots$) for the balanced descendant invariant
\begin{equation} \label{eqn:balanced}
    \langle \tau_{a} ^{n-b} \tau_{a+1}^b \rangle_{\aiinline{g}} = \max_{\textbf{e} \in \aiinline{E(g,n)}} D(\textbf{e})\,?
\end{equation}
\end{question}
A particularly simple regime might be the case $g \geq 2$ and $n \geq 3g-3$, where the dilaton equation allows to reduce all invariants to
\begin{equation} \label{eqn:dilaton}
    \langle \tau_1^{n-(3g-3)} \tau_2^{3g-3} \rangle_{\aiinline{g}} = \frac{(2g-3+n)!}{(5g-6)!} \langle \tau_2^{3g-3} \rangle_{\aiinline{g}}\,.
\end{equation}
\end{humanparagraph}
\begin{aiparagraph}
  
Questions of extremizing enumerative invariants within a finite parameter space are somewhat unusual in enumerative geometry, where one more typically studies generating functions or asymptotic behavior. It would be interesting to investigate analogous extremal questions for other families of intersection numbers, such as Hodge integrals or double Hurwitz numbers.
\end{aiparagraph}

\medskip
\begin{humanparagraph}
\noindent\textbf{Outline.} 
In Section \ref{sec:proof} we present the proof of Theorem~\ref{thm:main}, as it was found and stated by GPT-5 (for part (b)) and Gemini 3 Pro (for part (a)). We chose to leave the formulation largely untouched to allow the reader to see (for better or worse) the current level of proficiency in proof generation exhibited by these models.

Conversely, section \ref{sec:formalized} presents an alternative argument for Theorem~\ref{thm:main}, which can be read independently of Section \ref{sec:proof}. The treatment there splits the proof into:
\begin{itemize}
    \item a purely combinatorial optimization result (Theorem \ref{thm:main_theorem_paper}) which has been formalized in Lean \cite{lean4} based on the mathlib library \cite{mathlib},
    \item a geometric argument (Proposition  \ref{prop:D_satisfies_hypotheses}), with proof entirely human-written, proving that the descendant function $D$ from Theorem~\ref{thm:main} satisfies the assumptions of this optimization result.
\end{itemize}
In Appendix~\ref{sec:methods} we describe the various forms of AI assistance that went into the preparation of the present article. 
{In Appendix~\ref{sec:AI_attribution} we collect some proposals and tools for attributing AI contributions in mathematical papers.}

\medskip
\noindent\textbf{Acknowledgments.} 
We want to thank Javier Gómez-Serrano and Jaume De Dios Pont for inspiring and organizing, respectively, the OpenEvolve hackathon at which the present project was initiated. We are grateful to our colleagues who gave comments on Theorem \ref{thm:main}, in particular for confirming that to their knowledge it is not already covered in the literature. 

The proposals in Appendix \ref{sec:AI_attribution} were inspired by both a \href{https://ai-math.zulipchat.com/#narrow/channel/539992-Web-public-channel---AI-Math/topic/Best.20practices.20for.20incorporating.20AI.20etc.2E.20in.20papers/near/546518354}{public discussion thread started by Terry Tao} and conversations with colleagues, such as Jasper Dekoninck, Jeremy Feusi, Tim Gehrunger as well as the members of the  writing group of the \href{https://www.lorentzcenter.nl/index.php?pntType=ConPagina&id=2150&conBestandId=4187&pntHandler=DownloadAction}{Leiden declaration on the use of AI in mathematical research}, in particular Jarod Alper, Michael Harris, David Holmes, Ursula Martin and Jim Portegies, whom we thank warmly!

Finally, we are grateful to Jarod Alper, Honglu Fan, Jeremy Feusi, Tim Gehrunger, David Holmes and Jim Portegies for comments on various drafts of the present article.

The author was supported by SwissMAP and the  IMProofBench project received financial support from Google/DeepMind as part of their Google Gift program.

\end{humanparagraph}


\section{Proof of the main theorem}\label{sec:proof}
\begin{aiparagraph}
We write $d = \dim \Mbar_{g,n} = 3g - 3 + n$ throughout.

\subsection{\texorpdfstring{\humaninline{Maximal descendants}}{Maximal descendants}}

The following proof was generated by OpenAI's GPT-5 model. It is reproduced with only minimal formatting changes; see Appendix~\ref{sec:methods} for context. \humaninline{Note:} {The original AI output contained an error in the genus-0 formula stated in the remarks, giving $(n-3)!/\prod_i(2e_i-1)!!$ instead of the correct formula~\eqref{eq:genus0}. This has been corrected.}

\begin{proof}[Proof of Theorem \ref{thm:main} (b)]
We prove that among all $\ee$ with fixed sum, $D(\ee)$ attains its maximum at a balanced vector, i.e., when the entries differ by at most~1.

\medskip
\noindent\textbf{Step 1} (Two-point slice and its basic properties). Fix distinct indices $i \neq j$ and set
\[
M := \prod_{k \neq i,j} \psi_k^{e_k}, \qquad q := e_i + e_j.
\]
For $t = 0, 1, \ldots, q$, consider the one-variable sequence
\[
S_t := \int_{\Mbar_{g,n}} \psi_i^t \, \psi_j^{q-t} \, M.
\]
By the natural action of the symmetric group permuting the markings on $\Mbar_{g,n}$, the classes $\psi_i$ and $\psi_j$ are exchanged by an automorphism that leaves $M$ invariant; hence
\[
S_t = S_{q-t} \qquad \text{(palindromicity)}.
\]

\medskip
\noindent\textbf{Step 2} (Khovanskii--Teissier log-concavity). It is standard that each $\psi_i$ is nef on $\Mbar_{g,n}$. For nef classes $A, B$ and fixed nef classes $H_3, \ldots, H_d$ on a projective (orbifold) variety of dimension $d$, the Khovanskii--Teissier inequalities \humaninline{as described in Variant 1.6.2 of \cite{MR2095471}} (equivalently, the mixed Hodge--Riemann bilinear relations) give the discrete log-concavity
\[
\left( \int A^t B^{q-t} H_3 \cdots H_d \right)^2 \geq \left( \int A^{t-1} B^{q-t+1} H_3 \cdots H_d \right) \left( \int A^{t+1} B^{q-t-1} H_3 \cdots H_d \right)
\]
for all $1 \leq t \leq q - 1$. Applying this with $A = \psi_i$, $B = \psi_j$, $H_3 \cdots H_d = M$, we obtain
\[
S_t^2 \geq S_{t-1} S_{t+1} \qquad (1 \leq t \leq q - 1) \qquad \text{(log-concavity)}.
\]

\medskip
\noindent\textbf{Step 3} (Monotonicity towards the middle). For a positive log-concave sequence, the ratios $R_t := S_{t+1}/S_t$ are weakly decreasing in $t$. Using palindromicity,
\[
R_{q-t-1} = \frac{S_{q-t}}{S_{q-t-1}} = \frac{S_t}{S_{t+1}} = \frac{1}{R_t}.
\]
Hence for $t \leq \lfloor q/2 \rfloor - 1$ we have $R_t \geq R_{q-t-1} = 1/R_t$, so $R_t \geq 1$, i.e.,
\[
S_{t+1} \geq S_t \quad \text{for } t < \frac{q}{2}, \qquad S_{t-1} \geq S_t \quad \text{for } t > \frac{q}{2}.
\]
In words: along the two-point slice $\{(t, q-t) : 0 \leq t \leq q\}$, the values increase up to the middle and then decrease.

\medskip
\noindent\textbf{Step 4} (The balancing step). Suppose $e_i \geq e_j + 2$. Then $e_i > q/2$. By the monotonicity just proved,
\[
D(e_1, \ldots, e_i, \ldots, e_j, \ldots, e_n) = S_{e_i} \leq S_{e_i - 1} = D(e_1, \ldots, e_i - 1, \ldots, e_j + 1, \ldots, e_n).
\]
Thus, whenever two entries differ by at least 2, transferring one unit from the larger to the smaller weakly increases the value of the integral.

\medskip
\noindent\textbf{Step 5} (Conclusion by iteration). Starting from any $\ee \in E(g,n)$, repeatedly apply the balancing step to any pair $(i,j)$ with $|e_i - e_j| \geq 2$. This process terminates at a balanced vector $\ee^*$ (all entries differ by at most 1), and along the way the value of $D$ never decreases. Hence
\[
D(\ee) \leq D(\ee^*) \quad \text{for some balanced } \ee^* \in E(g,n).
\]
In particular, the maximum of $D$ on $E(g,n)$ is achieved at a balanced vector.
\end{proof}

\begin{remark}
The log-concavity in Step~2 is a special case of the Khovanskii--Teissier (or Alexandrov--Fenchel) inequalities for nef classes; it can be proved by restricting to a general complete intersection surface and applying the Hodge index theorem, or via the mixed Hodge--Riemann bilinear relations. The reduction from the Deligne--Mumford stack to a smooth projective variety can be made using a finite level-structure cover; intersection numbers scale by the degree of the cover, so inequalities are preserved.
\end{remark}

\subsection{\texorpdfstring{\humaninline{Minimal descendants}}{Minimal descendants}}

The minimum claim follows from the observation that \humaninline{iteratively} concentrating exponents on a single marking yields \humaninline{a non-increasing sequence of descendant integrals}, combined with the known formula for one-point integrals.

\begin{proof}
We employ the logic of Part (b) in reverse. Let $\ee = (e_1, \ldots, e_n) \in E(g,n)$ be an arbitrary exponent vector. We wish to show that $D(\ee)$ is bounded below by the value of the concentrated vector.

Let $k$ be an index such that $e_k = \max_{1 \leq i \leq n} e_i$. If $\ee$ is not fully concentrated on the $k$-th marking, there exists an index $j \neq k$ such that $e_j > 0$. Consider the pair of indices $(k, j)$ and let $q = e_k + e_j$. Since $e_k$ is maximal in $\ee$, we have $e_k \geq e_j$, which implies $e_k \geq q/2$.

Recall from Step~3 of the proof of part (b) that the sequence $S_t$ (representing the integral as a function of the exponent at the $k$-th marking) is non-increasing for $t \geq q/2$. Consequently, shifting the entire weight from the $j$-th marking to the $k$-th marking does not increase the value of the invariant:
\[
D(\ldots, e_k, \ldots, e_j, \ldots) \geq D(\ldots, e_k + e_j, \ldots, 0, \ldots).
\]
By iteratively applying this procedure for all $j \neq k$ with $e_j > 0$, we transform $\ee$ into the concentrated vector 
$$\ee_{\min} = \humaninline{(0, \ldots, 0, d, 0, \ldots, 0),}$$
(where the weight $d = 3g-3+n$ is at index $k$) while ensuring the value of the integral never increases. Thus $D(\ee) \geq D(\ee_{\min})$.

It remains to calculate the value of this minimum. The integral is given by $\langle \tau_d \tau_0^{n-1} \rangle_g$, where $\tau_0$ denotes the fundamental class. We apply the string equation
\[
\langle \tau_0 \prod_{i=1}^m \tau_{a_i} \rangle_g = \sum_{j=1}^m \langle \tau_{a_j-1} \prod_{i \neq j} \tau_{a_i} \rangle_g
\]
repeatedly to remove the $n-1$ insertions of $\tau_0$.
\humaninline{This works for $g \geq 1$; for $g=0$ we only remove $n-3$ such insertions, to avoid falling into the unstable range $n \leq 2$.} In each step, the term on the right-hand side vanishes unless we act on the single non-zero exponent (since $\tau_{-1}$ is zero). We therefore obtain the reduction
\[
\langle \tau_{3g-3+n} \tau_0^{n-1} \rangle_g = \langle \tau_{3g-4+n} \tau_0^{n-2} \rangle_g = \cdots = \langle \tau_{3g-2} \rangle_g.
\]
The value of the one-point integral $\langle \tau_{3g-2} \rangle_g = 1/(24^g g!)$ is a classical result following from the Witten--Kontsevich theorem (specifically, the coefficient of $t^g$ in the partition function restricted to one marked point). \humaninline{See in particular the discussion and references around equation \eqref{eqn:onepoint}.}
\end{proof}
\end{aiparagraph}

\section{Formalized proof via an abstract optimization theorem}\label{sec:formalized}

\begin{aiparagraph}
This section presents an alternative approach to proving Theorem~\ref{thm:main}, based on a Lean 4 formalization available at \href{https://schmittj.github.io/balanced-vectors-blueprint/index.html}{the project blueprint}. The key insight is to separate the \emph{abstract optimization argument}---which applies to any function satisfying certain axioms---from the \emph{geometric input} specific to $\psi$-class intersection numbers. This modular structure clarifies the logical dependencies and yields a fully machine-verified proof of the combinatorial core.

Throughout this section, we work with \emph{weak compositions}: for integers $n \geq 1$ and $d \geq 0$, define
\[
E(n, d) = \left\{ \ee = (e_1, \ldots, e_n) \in \ZZ_{\geq 0}^n : \sum_{j=1}^n e_j = d \right\}.
\]
For $\ee \in E(n,d)$ and distinct indices $i, j$, we write $\ee - \delta_i + \delta_j$ for the vector obtained by decreasing $e_i$ by $1$ and increasing $e_j$ by $1$ (defined when $e_i \geq 1$).

\begin{theorem}[Optimization Theorem - {\href{https://schmittj.github.io/balanced-vectors-blueprint/sect0001.html\#thm:main}{\doublecheck}},{ \href{https://schmittj.github.io/balanced-vectors-blueprint/docs/BalancedVectors.html\#main_theorem_paper}{LEAN}}]\label{thm:main_theorem_paper} 
Let $D : E(n, d) \to \QQ$ be a function satisfying:
\begin{enumerate}
    \item[\textup{(S)}] \textbf{Symmetry}: $D(\ee \circ \sigma) = D(\ee)$ for all permutations $\sigma \in S_n$.
    \item[\textup{(LC)}] \textbf{Log-concavity}: For all $\ee \in E(n,d)$ and distinct $i, j$ with $e_i, e_j \geq 1$,
    \[
    D(\ee)^2 \geq D(\ee - \delta_i + \delta_j) \cdot D(\ee + \delta_i - \delta_j).
    \]
    \item[\textup{(P)}] \textbf{Strict positivity}: $D(\ee) > 0$ for all $\ee \in E(n,d)$.
\end{enumerate}
Then:
\begin{enumerate}
    \item[\textup{(a)}] $D$ achieves its maximum on a balanced vector $($where $|e_i - e_j| \leq 1$ for all $i, j)$.
    \item[\textup{(b)}] $D$ achieves its minimum on a concentrated vector $($where $\ee = d \cdot \delta_k$ for some $k)$.
\end{enumerate}
\end{theorem}

The proof \humaninline{is presented in Sections \ref{sect:slice} to \ref{sect:concentration}. Assuming the theorem for now, } we verify that the descendant integral satisfies the hypotheses \humaninline{above}.

\begin{proposition}\label{prop:D_satisfies_hypotheses}
Let $g, n \in \ZZ_{\geq 0}$ with $2g - 2 + n > 0$ and $n \geq 1$. The function
\[
D : E(n, 3g-3+n) \to \QQ, \qquad D(\ee) = \int_{\Mbar_{g,n}} \psi_1^{e_1} \cdots \psi_n^{e_n}
\]
satisfies conditions \textup{(S)}, \textup{(LC)}, and \textup{(P)} of Theorem~\ref{thm:main_theorem_paper}.
\end{proposition}
\end{aiparagraph}

\begin{humanparagraph}
\begin{proof}
For the symmetry \textup{(S)}, we observe that for any $\sigma \in S_n$ with inverse $\tau = \sigma^{-1}$ there exists a natural isomorphism
\begin{equation}
    \phi_\sigma : \Mbar_{g,n} \to \Mbar_{g,n}, (C, p_1, \ldots, p_n) \mapsto (C, p_{\sigma(1)}, \ldots, p_{\sigma(n)})\,,
\end{equation}
which satisfies $\phi_\sigma^* \psi_i = \psi_{\tau(i)}$. Using the projection formula together with the pushforward relation $(\phi_\sigma)_* [\Mbar_{g,n}] = [\Mbar_{g,n}]$ of the fundamental class, we obtain:
\begin{align*}
    D(\textbf{e} \circ \sigma) &= \int_{\Mbar_{g,n}} \psi_1^{e_{\sigma(1)}} \cdots \psi_n^{e_{\sigma(n)}} = \int_{\Mbar_{g,n}} \psi_{\tau(1)}^{e_1} \cdots \psi_{\tau(n)}^{e_n}\\
    &= \operatorname{deg}\left(\psi_{\tau(1)}^{e_1} \cdots \psi_{\tau(n)}^{e_n} \frown [\Mbar_{g,n}] \right)\\
    &= \operatorname{deg}\left( (\phi_\sigma)^* (\psi_{1}^{e_1} \cdots \psi_{n}^{e_n}) \frown (\phi_\sigma)_* [\Mbar_{g,n}] \right)\\
    &= \operatorname{deg}\left( (\phi_\sigma)_* (\psi_{1}^{e_1} \cdots \psi_{n}^{e_n} \frown [\Mbar_{g,n}]) \right)\\
    &=\int_{\Mbar_{g,n}} \psi_1^{e_1} \cdots \psi_n^{e_n} = D(\textbf{e})\,.
\end{align*}
To prove property \textup{(LC)} we first recall from Variant 1.6.2 of \cite{MR2095471} that for $X$ an irreducible complete scheme of dimension $d \geq 2$ and nef divisor classes
\[
\alpha_1, \alpha_2, \beta_1, \ldots, \beta_{d-2} \in N_1(X)_{\mathbb{R}}
\]
we have
\begin{equation} \label{eqn:positivity_inequality}
\left(\alpha_1 \cdot \alpha_2 \cdot \beta_1 \cdots \beta_{d-2} \right)^2 \geq \left(\alpha_1^2 \cdot \beta_1 \cdots \beta_{d-2} \right) \cdot \left(\alpha_2^2 \cdot \beta_1 \cdots \beta_{d-2} \right)\,.
\end{equation}
Now fix $d=3g-3+n$ and $\ee \in E(n,d)$ and distinct $i, j$ with $e_i, e_j \geq 1$, and choose $\widetilde \alpha_1 = \psi_j$, $\widetilde \alpha_2 = \psi_i$ and $\widetilde \beta_1, \ldots, \widetilde \beta_{d-2} \in \{\psi_1, \ldots, \psi_n\}$ such that
\[
\psi_1^{e_1} \cdots \psi_n^{e_n} = \widetilde \alpha_1 \cdot \widetilde \alpha_2 \cdot \widetilde \beta_1 \cdots \widetilde \beta_{d-2}\,.
\]
Note that this step is just a reordering of factors, using that $e_i, e_j \geq 1$ to isolate the two factors $\widetilde \alpha_1$ and $\widetilde \alpha_2$. Let $\varphi: X \to \Mbar_{g,n}$ be a finite flat cover of degree $M$ by an irreducible complete scheme $X$ (see e.g. \cite[Theorem 7.4.2]{MR2007376}). Write $\alpha_i = \varphi^* \widetilde \alpha_i$ and $\beta_i = \varphi^* \widetilde \beta_i$. Since all $\psi$-classes are nef (\cite[Chapter XIV, Corollary 5.14]{MR2807457}) and this property is preserved under pullback by $\varphi$, we can apply \eqref{eqn:positivity_inequality} together with the projection formula and obtain:
\begin{align*}
    D(\ee)^2 &=\left(\int_{\Mbar_{g,n}} \widetilde \alpha_1 \cdot \widetilde \alpha_2 \cdot \widetilde \beta_1 \cdots \widetilde \beta_{d-2} \right)^2 = \frac{1}{M^2} \left( \int_X \alpha_1 \cdot \alpha_2 \cdot \beta_1 \cdots \beta_{d-2} \right)^2\\
    &\geq \left(\frac{1}{M} \int_X \alpha_1^2 \cdot \beta_1 \cdots \beta_{d-2} \right) \cdot \left(\frac{1}{M} \int_X \alpha_2^2 \cdot \beta_1 \cdots \beta_{d-2} \right)\\
    &= \left(\int_{\Mbar_{g,n}} \widetilde \alpha_1^2 \cdot \widetilde \beta_1 \cdots \widetilde\beta_{d-2} \right) \cdot \left(\int_{\Mbar_{g,n}} \widetilde \alpha_2^2 \cdot \widetilde \beta_1 \cdots \widetilde\beta_{d-2} \right)\\ &= D(\ee - \delta_i + \delta_j) \cdot D(\ee + \delta_i - \delta_j).
\end{align*}
Finally, the positivity of all descendant invariants $D(\ee)$ is well known, following e.g. from the recursive formulas for computing these invariants derived from the Witten--Kontsevich theorem \cite{Witten1991, Kontsevich1992}.
\end{proof}
\end{humanparagraph}

\begin{aiparagraph}
\begin{remark}
Theorem~\ref{thm:main}(a) and (b) follow immediately from Theorem~\ref{thm:main_theorem_paper} and Proposition~\ref{prop:D_satisfies_hypotheses}. The formula~\eqref{eqn:min} for the minimal value is then obtained by the string equation argument given in Section~\ref{sec:proof}.
\end{remark}

\subsection{Slice sequences} \label{sect:slice}

The proof of Theorem~\ref{thm:main_theorem_paper} proceeds by analyzing how $D$ varies along ``slices'' where only two coordinates change.

\begin{definition}[Slice sequence - \href{https://schmittj.github.io/balanced-vectors-blueprint/docs/BalancedVectors.html\#sliceSeq}{LEAN}]\label{def:sliceSeq}
Let $D : E(n,d) \to \QQ_{>0}$ satisfy \textup{(S)}, \textup{(LC)}, \textup{(P)}. Given $\ee \in E(n,d)$ and distinct indices $i \neq j$, define the \emph{slice sequence} $S : \{0, 1, \ldots, q\} \to \QQ_{>0}$ by
\[
S_t = D(\ee^{(t)}),
\]
where $q = e_i + e_j$ and $\ee^{(t)}$ is the composition agreeing with $\ee$ at all indices $k \notin \{i, j\}$, with $e^{(t)}_i = t$ and $e^{(t)}_j = q - t$.
\end{definition}

The key observation is that $S$ inherits strong structural properties from $D$.

\begin{lemma}[Palindromicity - \href{https://schmittj.github.io/balanced-vectors-blueprint/docs/BalancedVectors.html\#sliceSeq_palindromic}{\doublecheck, LEAN}]\label{lem:sliceSeq_palindromic}
The slice sequence satisfies $S_t = S_{q-t}$ for all $0 \leq t \leq q$.
\end{lemma}

\begin{proof}
The transposition $\sigma = (i \; j) \in S_n$ exchanges the $i$-th and $j$-th coordinates. Since $\ee^{(t)} \circ \sigma = \ee^{(q-t)}$, symmetry \textup{(S)} gives $S_t = D(\ee^{(t)}) = D(\ee^{(q-t)}) = S_{q-t}$.
\end{proof}

\begin{lemma}[Log-concavity - \href{https://schmittj.github.io/balanced-vectors-blueprint/docs/BalancedVectors.html\#sliceSeq_logconcave}{\doublecheck, LEAN}]\label{lem:sliceSeq_logconcave}
The slice sequence satisfies $S_t^2 \geq S_{t-1} \cdot S_{t+1}$ for all $1 \leq t \leq q - 1$.
\end{lemma}

\begin{proof}
For $1 \leq t \leq q - 1$, the composition $\ee^{(t)}$ has $e^{(t)}_i = t \geq 1$ and $e^{(t)}_j = q - t \geq 1$. Applying condition \textup{(LC)} to $\ee^{(t)}$ with indices $i, j$ yields
\[
D(\ee^{(t)})^2 \geq D(\ee^{(t)} - \delta_i + \delta_j) \cdot D(\ee^{(t)} + \delta_i - \delta_j) = D(\ee^{(t-1)}) \cdot D(\ee^{(t+1)}),
\]
which is exactly $S_t^2 \geq S_{t-1} \cdot S_{t+1}$.
\end{proof}

\subsection{Unimodality of log-concave palindromic sequences}

The following general lemma is the heart of the argument.

\begin{lemma}[Unimodality - \href{https://schmittj.github.io/balanced-vectors-blueprint/sect0001.html\#lem:unimodal}{\doublecheck}, \href{https://schmittj.github.io/balanced-vectors-blueprint/docs/BalancedVectors.html\#unimodal_of_logconcave_palindromic}{LEAN}]\label{lem:unimodal_of_logconcave_palindromic}
Let $S : \{0, 1, \ldots, q\} \to \QQ_{>0}$ be a positive sequence satisfying:
\begin{itemize}
    \item Palindromicity: $S_t = S_{q-t}$ for all $0 \leq t \leq q$.
    \item Log-concavity: $S_t^2 \geq S_{t-1} \cdot S_{t+1}$ for all $1 \leq t \leq q - 1$.
\end{itemize}
Then $S$ is unimodal with maximum at the center:
\begin{enumerate}
    \item[\textup{(i)}] $S_t \leq S_{t+1}$ whenever $2t < q$.
    \item[\textup{(ii)}] $S_t \leq S_{t-1}$ whenever $2t > q$.
\end{enumerate}
\end{lemma}

\begin{proof}
Define the ratio $R_t = S_{t+1}/S_t$ for $0 \leq t \leq q - 1$. Log-concavity implies that the ratios are weakly decreasing: from $S_t^2 \geq S_{t-1} S_{t+1}$ we obtain $S_t/S_{t-1} \geq S_{t+1}/S_t$, i.e., $R_{t-1} \geq R_t$.

Palindromicity gives a reflection identity for ratios:
\[
R_{q-1-t} = \frac{S_{q-t}}{S_{q-1-t}} = \frac{S_t}{S_{t+1}} = \frac{1}{R_t}.
\]
For $t$ with $2t < q - 1$, monotonicity of ratios gives $R_t \geq R_{q-1-t} = 1/R_t$, hence $R_t^2 \geq 1$. Since $R_t > 0$, we conclude $R_t \geq 1$, i.e., $S_{t+1} \geq S_t$. This proves (i).

Part (ii) follows from (i) by palindromicity: if $2t > q$, then $2(q - t) < q$, so $S_{q-t} \leq S_{q-t+1}$, which by palindromicity becomes $S_t \leq S_{t-1}$.
\end{proof}

\subsection{The balancing step}

\begin{lemma}[Balancing increases $D$ - \href{https://schmittj.github.io/balanced-vectors-blueprint/docs/BalancedVectors.html\#balancing_increases_D}{\doublecheck, LEAN}]\label{lem:balancing_increases_D}
Let $D$ satisfy \textup{(S)}, \textup{(LC)}, \textup{(P)}. If $\ee \in E(n,d)$ has $e_i \geq e_j + 2$ for some $i \neq j$, then
\[
D(\ee) \leq D(\ee - \delta_i + \delta_j).
\]
\end{lemma}

\begin{proof}
Consider the slice sequence $S$ for $\ee$ with indices $i, j$. The original composition $\ee$ corresponds to $S_{e_i}$, and the modified composition $\ee - \delta_i + \delta_j$ corresponds to $S_{e_i - 1}$.

Since $e_i \geq e_j + 2$, we have $e_i > (e_i + e_j)/2 = q/2$, so $2e_i > q$. By the unimodality lemma (part (ii)), $S_{e_i} \leq S_{e_i - 1}$, which gives $D(\ee) \leq D(\ee - \delta_i + \delta_j)$.
\end{proof}

\begin{proof}[Proof of Theorem~\ref{thm:main_theorem_paper}(a)]
Starting from any $\ee \in E(n,d)$, repeatedly apply the balancing step: whenever there exist $i, j$ with $e_i \geq e_j + 2$, replace $\ee$ by $\ee - \delta_i + \delta_j$. By Lemma~\ref{lem:balancing_increases_D}, each step weakly increases $D$.

This process terminates in finitely many steps because the \emph{imbalance} $\sum_k e_k^2$ strictly decreases with each modification (one verifies $(a-1)^2 + (b+1)^2 < a^2 + b^2$ when $a \geq b + 2$). The process terminates precisely when no pair $(i, j)$ satisfies $|e_i - e_j| \geq 2$, i.e., when $\ee$ is balanced.

Thus every $\ee$ can be transformed to a balanced $\ee^*$ with $D(\ee) \leq D(\ee^*)$. In particular, the maximum of $D$ over the finite set $E(n,d)$ is achieved at a balanced vector.
\end{proof}

\subsection{The concentrating step} \label{sect:concentration}

\begin{lemma}[Concentrating decreases $D$ - \href{https://schmittj.github.io/balanced-vectors-blueprint/docs/BalancedVectors.html\#concentrating_decreases_D}{\doublecheck, LEAN}]\label{lem:concentrating_decreases_D}
Let $D$ satisfy \textup{(S)}, \textup{(LC)}, \textup{(P)}. If $\ee \in E(n,d)$ has $e_j \geq 1$ and $e_j \leq e_i$ for some $i \neq j$, then
\[
D(\ee - \delta_j + \delta_i) \leq D(\ee).
\]
\end{lemma}

\begin{proof}
Consider the slice sequence $S$ for $\ee$ with indices $j, i$ (note the order). The original composition corresponds to $S_{e_j}$, and the modified composition $\ee - \delta_j + \delta_i$ corresponds to $S_{e_j - 1}$.

Since $e_j \leq e_i$, we have $2e_j \leq e_j + e_i = q$. If $2e_j < q$, then by unimodality (part (i)), $S_{e_j - 1} \leq S_{e_j}$, giving the result. If $2e_j = q$, then $e_j = e_i$ and palindromicity gives $S_{e_j - 1} = S_{q - (e_j - 1)} = S_{e_j + 1}$; unimodality then gives $S_{e_j - 1} = S_{e_j + 1} \leq S_{e_j}$.
\end{proof}

\begin{proof}[Proof of Theorem~\ref{thm:main_theorem_paper}(b)]
Starting from any $\ee \in E(n,d)$, repeatedly apply the concentrating step: let $i$ be an index achieving the maximum $\max_k e_k$, and if there exists $j \neq i$ with $e_j \geq 1$, replace $\ee$ by $\ee - \delta_j + \delta_i$. By Lemma~\ref{lem:concentrating_decreases_D}, each step weakly decreases $D$.

This process terminates because the maximum entry $\max_k e_k$ strictly increases with each step (and is bounded by $d$). The process terminates when all mass is concentrated at a single index, i.e., when $\ee = d \cdot \delta_k$ for some $k$.

Thus every $\ee$ can be transformed to a concentrated $\ee_*$ with $D(\ee_*) \leq D(\ee)$. In particular, the minimum of $D$ is achieved at a concentrated vector.
\end{proof}

This completes the proof of Theorem~\ref{thm:main_theorem_paper}.
\end{aiparagraph}

\appendix

\section{Methodology: AI-assisted discovery and proof}\label{sec:methods}
\begin{aiparagraph}

This appendix documents the role of AI systems in discovering the conjecture and producing the proof of Theorem~\ref{thm:main}(b).

\subsection{Discovery of the conjecture}

The initial observation arose during a hackathon on AI-assisted mathematics organized by Jaume de Dios Pont at ETH Z\"urich, focused on OpenEvolve \humaninline{\cite{openevolve} --- an open-source implementation of the AlphaEvolve system \cite{novikov2025alphaevolvecodingagentscientific} for evolutionary search over functions. Inspired by recent applications of such systems for mathematical exploration \cite{georgiev2025mathematicalexplorationdiscoveryscale}}, the author chose as a test problem the optimization of descendant integrals on $\Mbar_{g,n}$ for fixed $g$ and $n$.

While setting up the project using Claude Code (running Anthropic's Claude Opus~4.5), the AI assistant observed from initial computations:
\begin{quote}
\itshape
``The interesting pattern is that for $g=0$, spreading the exponents seems to give larger values (e.g., $[1,1,1,0,0,0] = 6$ vs.\ $[3,0,0,0,0,0] = 1$).''
\end{quote}
In the configuration file for OpenEvolve, the model included:
\begin{quote}
\ttfamily\small
Mathematical insight for genus 0:\\
- The integral equals (n-3)! / (e\_1! * e\_2! * ... * e\_n!)\\
- This is maximized when the exponents are as evenly distributed as possible\\
- For example, with n=6 and dim=3, [1,1,1,0,0,0] gives 6, while [3,0,0,0,0,0] gives 1\\
For higher genus:\\
- The formula is more complex (Witten-Kontsevich theorem)\\
- General pattern: spreading exponents tends to give larger values\\
- But the optimal distribution may differ from genus 0
\end{quote}
The evolutionary search then quickly converged to balanced vectors as optima.

\subsection{Submission to IMProofBench}

Based on these observations, the author formulated the conjecture (for the maximum; the minimum was added later) and submitted it to IMProofBench \cite{IMProofBench}, a benchmark for evaluating AI systems on research-level mathematical proof generation.

The problem was evaluated against several large language models in an agentic framework with access to SageMath and other computational tools. Results varied significantly:
\begin{itemize}
    \item Several models produced inconclusive responses: proofs for genus~0 only, or vague suggestions for higher genus.
    \item Some models claimed false counterexamples.
    \item At least one model cited a reference that appears to be entirely hallucinated:
    \begin{quote}
    \itshape
    ``Zhou, Y. (2018). A proof of the concavity conjecture for intersection numbers of $\psi$-classes. Journal of Differential Geometry, 110(2), 361--389.''
    \end{quote}
    \humaninline{This paper does not exist; the above volume of JDG instead features the paper \cite{MR3861813} on pages 345--377, being the last paper in this issue.}
    \item Some models used the provided SageMath tools to verify the conjecture computationally for small values of $g$ and $n$.
    \item Multiple evaluations of \humaninline{the AI models o3, GPT-5 and GPT-5 Pro} converged on the proof strategy using nefness and Khovanskii--Teissier log-concavity presented in Section~\ref{sec:proof}.
\end{itemize}

The proof in Section~\ref{sec:proof} is taken from one of the GPT-5 evaluations. Interestingly, the single evaluation of the problem against GPT-5.1 did not solve the problem.
\end{aiparagraph}

\subsection{Partial formalization and Lean Blueprint}
\begin{humanparagraph}
Based on discussions with colleagues and inspired by \href{https://emilyriehl.github.io/files/norms-hardy.pdf#page=33}{a proposal of Emily Riehl}, we also decided to formalize at least parts of the proof found by GPT-5, as presented in Section \ref{sec:formalized}. To our knowledge, the prerequisite results in algebraic geometry going into the proof (divisors on algebraic schemes, intersection numbers, nef classes, moduli spaces of curves, the Witten--Kontsevich theorem, etc) are not yet formalized in Lean. Thus we decided to split the proof in a purely combinatorial optimization result (Theorem \ref{thm:main_theorem_paper}), and a geometric part proved non-formally (Theorem \ref{prop:D_satisfies_hypotheses}).

As for the work of formalization itself: the author has not had any prior experience with Lean or the mathlib. For all manipulations of \aiinline{the} respective \texttt{.lean} file we used Claude Code (Opus 4.5) which loaded the \href{https://github.com/cameronfreer/lean4-skills}{Lean 4 Skills} collection developed by Cameron Freer. For particularly tricky problems or intermediate planning we also copy-pasted summaries of the current formalization state by Claude together with the latest version of the Lean file into ChatGPT 5.2, pasting its answer back to Claude Code.

Progress was slow, but steady, and after several hours of work the Lean file compiled without errors or missing proofs. During this time, the contribution of the author was confined to:
\begin{itemize}
    \item providing an initial description of the result to be formalized, with existing AI-generated proof for context,
    \item giving very few pointers on possible mathematical issues,
    \item occasionally soliciting updates, general management of the order of work, and interventions via GPT 5.2 commentary when progress stagnated.
\end{itemize}
In particular, we did not edit a single line of \texttt{.lean} code. A transcript of the conversation can be found \href{https://johannesschmitt.gitlab.io/balanced_vectors_lean_conversation.txt}{here}.

After the formalization was complete, the Lean code was made available via a GitHub repository using the \href{https://github.com/PatrickMassot/leanblueprint}{Lean blueprint} plugin developed by Patrick Massot and collaborators. Again, with exceptions of a few clicks on the GitHub website and  debugging information provided by the author, all files were created and edited by Claude Code.

The resulting annotated file \texttt{balanced-vectors-blueprint.lean} was then converted back to the human-readable proofs in Section \ref{sec:formalized} by \href{https://claude.ai/share/475ce24b-1d24-44c1-883d-2dd3ae1181d9}{Claude Opus 4.5}. In addition, we want to thank Eric Vergo, who refactored the original file \href{https://github.com/e-vergo/Balanced_Vectors}{in a separate repository} to make it easier to quickly inspect the formalization of the main theorem, without having to trace through the entire Lean file.

The whole process of creating the original Lean formalization and blueprint above took about a day to complete.
\end{humanparagraph}

\subsection{Paper preparation}
\begin{aiparagraph}
This paper was prepared with assistance from Claude (Anthropic), which helped structure the document, identify issues in the AI-generated proof (including the incorrect genus-0 formula in the original remarks), and draft portions of this methodology section.
\end{aiparagraph}
\begin{humanparagraph}
Below we collect some links to shared example conversations with AI models, which include the prompts used to get the relevant outputs:
\begin{itemize}
    \item \href{https://claude.ai/share/6dd9c266-57d7-4a89-82a4-2fb996c86bef}{Paper drafting session} (Claude Opus 4.5)
    \item \href{https://claude.ai/public/artifacts/213f5978-f358-4865-a8ca-c3ef784e0acb}{Summary of the IMProofBench project} and \href{https://claude.ai/public/artifacts/c822fac4-8908-4c64-b0cc-645c4af6a8dd}{literature overview} created to provide context for the paper drafting (Claude Opus 4.5 + Research Mode)
    \item \href{https://gemini.google.com/share/ab164342674e}{Reference for \eqref{eqn:onepoint}} (Gemini 3 Pro), \href{https://chatgpt.com/share/693fe7a9-7da4-8006-9f21-183f991a7efb}{Reference for finite cover of $\Mbar_{g,n}$ by scheme} (ChatGPT 5.2)
    \item \href{https://gemini.google.com/share/591699ebcecc}{Proof of Theorem~\ref{thm:main}(a)} (Gemini 3 Pro)
    \item \href{https://claude.ai/share/a1fcf89b-6a34-439f-97ea-46666be71c13}{Writing code to verify equation \eqref{eqn:dilaton} in examples } (Claude Opus 4.5)
    \item \href{https://claude.ai/share/475ce24b-1d24-44c1-883d-2dd3ae1181d9}{Writing the LaTeX proofs of Section \ref{sec:formalized}} based on the existing \href{https://schmittj.github.io/balanced-vectors-blueprint/index.html}{Lean Blueprint file} (Claude Opus 4.5)
    \item Examples of Lean assistance provided by ChatGPT 5.2 (\href{https://chatgpt.com/share/694000dd-5618-8006-ab20-924a00bc0e8e}{Conversation 1}, \href{https://chatgpt.com/share/69400106-5118-8006-8b0f-847d6e9c34b8}{Conversation 2}, \href{https://chatgpt.com/share/6940011d-bfa8-8006-8101-683893a60104}{Conversation 3}, \href{https://chatgpt.com/share/69400132-bcac-8006-80ad-72b645ed8185}{Conversation 4})
    \item Creating LaTeX environments for labelling human and AI paragraphs : \href{https://claude.ai/share/27e08d7b-5118-481d-b374-05192aab1ff4}{Version 1}, \href{https://claude.ai/share/969fdd46-3991-44c5-a409-5f79fb2be0dc}{Version 2}, \href{https://claude.ai/share/8a36b89e-79ae-4911-aff7-94960821d350}{Version 3} (Claude Opus 4.5)
\end{itemize}
The first raw draft of the paper was also submitted to Claude Code, \href{https://gemini.google.com/share/65ff8387279c}{Gemini 3 Pro} and \href{https://chatgpt.com/share/693833a5-3e50-8006-b3fd-6592f456f1fb}{ChatGPT 5.1} for feedback and proofreading. This resulted in several improvements, e.g.
\begin{itemize}
    \item Restricting to $n>0$ in Theorem~\ref{thm:main} to avoid statements about mimima/maxima of the empty set, and adding the term $\tau_0^{n-1}$ in that theorem, for clarity.
    \item Pointing out that in the proof of Theorem~\ref{thm:main}(a), the case $g=0$ cannot reduce all the way to $n=1$ via the string equation (see the corresponding human correction).
\end{itemize}
The usage of the described AI models was covered by subscriptions in the form of Claude Max (\$100/month), ChatGPT Plus (\$20/month) and Google AI Pro (\$20/month).
\end{humanparagraph}

\section{Best practices for attribution of AI use} \label{sec:AI_attribution}
\begin{humanparagraph}
After some first pioneering papers \cite{ghrist2024latticevaluedbottleneckduality, raz2025euleraiunifyingformulas}, recent months have seen a steadily growing list of articles  \cite{avvakumov2025tensorrankdeterminantperiodic, jang2025pointconvergencenesterovsaccelerated, alexeev2025forbiddensidonsubsetsperfect, ivanisvili2025counterexamplemajorityoptimalitynicd, datar2025complexmongeampereequationapplication, fortnow2025searchversusdecisionmathsfs2mathsfp, dobriban2025solvingresearchproblemmathematical, bubeck2025earlyscienceaccelerationexperiments}  using Large Language Models both to study mathematical questions and writing up the results. {See e.g. the \href{https://github.com/seewoo5/awesome-ai-for-math/blob/main/subjects/llm.md}{\texttt{awesome-ai-for-math}} repository maintained by Seewoo Lee for a list of examples.}

Given the potential impact of these tools, it is important for the mathematical community to develop some norms and best practices for attributing their use in our work. Below we propose some ideas for such measures, and apply these to the present article.

\vspace{5pt}
\noindent \textbf{Methodology section}\\
Mathematical papers that benefited from using LLMs, Deep-Learning, proof assistants, or other software tools should include a section clearly attributing this usage. This section should include or reference any sample prompts, data repositories or code files that could help others to examine and reproduce the employed methods, and to adapt them for their own research.

\vspace{5pt}
\noindent \textbf{Transparency about writing tools}\\
When using AI tools for writing first drafts of parts of the paper (e.g. literature overviews, abstract or content summary, proof sketches, etc), all paragraphs with substantial AI contributions could be typeset in a visually distinctive way. 
\end{humanparagraph}
\begin{aiparagraph}
This paragraph was generated by Claude Opus 4.5 (Anthropic) to demonstrate the visual formatting conventions proposed in this appendix. The orange margin bar and label allow readers to immediately identify AI-authored content, complementing the blue markers used for human-written passages.
\end{aiparagraph}

\begin{humanparagraph}
A first LaTeX template for such sidebars (created joint\aiinline{ly} with Claude Opus 4.5) is available here:
\begin{center}
    \url{https://johannesschmitt.gitlab.io/attribution-macros.tex}
\end{center}

\vspace{5pt}
\noindent \textbf{Contextualization for non-expert readers}\\
Given the current public interest in AI applications for academic research, results obtained with the help of AI tend to gain increased levels of attention. In particular, those papers are exposed to  an audience outside of the specialized area in which the result would usually be received. 
This creates some incentive, both for the author and the developers of the relevant AI model, to overstate the importance of the presented result. Thus we believe it is the responsibility of the author to contextualize the significance of the presented work for non-experts, ideally in a somewhat prominent place (as with the Author's Note on the first page of this article).
\end{humanparagraph}

\bibliographystyle{alpha}
\bibliography{references}

\end{document}